\documentclass[envcountsect]{llncs}
\usepackage[utf8]{inputenc}
\usepackage{amssymb,amsmath,mathrsfs}
\usepackage[english]{babel}
\usepackage[unicode,unicode=true,bookmarks=false]{hyperref}
\usepackage[shortlabels]{enumitem}
\usepackage{mathtools}
\usepackage{tikz}
\usetikzlibrary{automata, arrows.meta, positioning}

\usepackage{environ}

\newcommand{\repeattheorem}[1]{%
	\begingroup
	\renewcommand{\thetheorem}{\ref{#1}}%
	\expandafter\expandafter\expandafter\theorem
	\csname reptheorem@#1\endcsname
	\endtheorem
	\endgroup
}

\NewEnviron{reptheorem}[1]{%
	\global\expandafter\xdef\csname reptheorem@#1\endcsname{%
		\unexpanded\expandafter{\BODY}%
	}%
	\expandafter\theorem\BODY\unskip\label{#1}\endtheorem
}

\newcommand{\NN}{\mathbb{N}}

\newcommand{\PrA}{\mathop{\mathbf{PrA}}\nolimits}

\newcommand{\BA}{\mathop{\mathbf{BA}}\nolimits}

\newcommand{\Th}{\mathop{\mathbf{Th}}\nolimits}

\newcommand{\Digit}{\mathop{Digit}\nolimits}

\makeatletter

\renewcommand{\epsilon}{\varepsilon}
\renewcommand{\phi}{\varphi}
\newcommand{\sref}[2]{\hyperref[#2]{#1 \ref*{#2}}}
\newcommand{\dref}[2]{\hyperref[#2]{ #1 }}

\newcommand{\Ac}{\mathcal{A}}
\newcommand{\Bc}{\mathcal{B}}

\newcommand{\Kc}{\mathcal{K}}
\newcommand{\Lc}{\mathcal{L}}

\newcommand{\equivdef}{\xLeftrightarrow{\mathit{def}}}

\newcommand{\ra}{\rightarrow}

\newcommand{\Lra}{\Leftrightarrow}

\spnewtheorem{hyp}{Conjecture}[section]{\bfseries}{\itshape}
\spnewtheorem{ex}{Example}{\bfseries}{\itshape}
\spnewtheorem{stm}{Statement}[section]{\bfseries}{\itshape}

\makeatletter
\newcommand{\dotminus}{\mathbin{\text{\@dotminus}}}

\newcommand{\@dotminus}{%
	\ooalign{\hidewidth\raise1ex\hbox{.}\hidewidth\cr$\m@th-$\cr}%
}
\makeatother

\begin{document}
	\author{Alexander Zapryagaev\thanks{The publication was prepared within the framework of the Academic Fund Program at HSE University (grant 21-04-027).}}
	\title{On Interpretations in B\"uchi Arithmetics}
	\institute{National Research University Higher School of Economics, 6, Usacheva Str., Moscow, 119333, Russian Federation}
	
	\maketitle
	
	\begin{abstract}
		B\"uchi arithmetics $\BA_n$, $n\ge 2$, are extensions of Presburger arithmetic with an unary functional symbol $V_n(x)$ denoting the largest power of $n$ that divides $x$. Definability of a set in $\BA_n$ is equivalent to its recognizability by a finite automaton receiving numbers in their $n$-ary expansion. We show that B\"uchi arithmetics $\BA_n$ and $\BA_m$ are bi-interpretable for any $n,m$. Furthermore, we establish that any interpretation of some structure in $\BA_n$ is isomorphic to some one-dimensional interpretation; however, this isomorphism must not be $\BA_n$-definable.
	\end{abstract}
	
	\section{Preliminaries}
	
	A B\"uchi arithmetic $\BA_n$, $n\ge 2$, is the theory $\Th(\NN,=,+,V_n)$ where $V_n$ is an unary functional symbol such that $V_n(x)$ is the largest power of $n$ that divides $x$ ($V_n(0)$ is defined to be $0$). They form a family of decidable extensions of Presburger arithmetic $\PrA=\Th(\NN,=,+)$ \cite{presburger}.
	
	These theories were proposed by J.~B\"uchi \cite{bc} in 1960 in order to describe the recognizability of sets of natural numbers by finite automata through logic. This relation is made explicit in the following classic result by V\'eronique Bruy\`ere \cite{bruyere,bv}:
	
	\begin{theorem}\label{there}
		Let $\varphi(x_1,\ldots,x_m)$ be a $\BA_n$-formula. Then there is an effectively constructed automaton $\Ac$ such that $(a_1,\ldots,a_m)$ is accepted by $\Ac$ iff $\NN\models\varphi(a_1,\ldots,a_m)$.
		
		Contrariwise, let $\Ac$ be a finite automaton working on $m$-tuples of $n$-ary natural numbers. Then there is an effectively constructed $\BA_n$-formula $\varphi(x_1,\ldots,x_m)$ such that $\NN\models\varphi(a_1,\ldots,a_m)$ iff $(a_1,\ldots,a_m)$ is accepted by $\Ac$. Furthermore, this formula is of complexity class not surpassing $\Sigma_2$ \cite{haase}.
	\end{theorem}.
	
	We consider the multi-dimensional non-paramentic interpretations \cite[pp.~20--21]{tarskimostowski} in $\BA_n$. An \emph{interpretation} $\iota$ of some first-order language $\Kc$ in a structure $\mathfrak{B}$ of language $\Lc$ is the translation of $\Kc$-formulas into $\Lc$-formulas defined the following $\Lc$-formulas:
	
	\begin{enumerate}
		\item $D_{\iota}(y)$ defining the set $\mathbf{D}_{\iota}\subseteq\mathfrak{B}$ (the \emph{domain formula});			
		\item $P_{\iota}(x_1,\ldots,x_n)$ for predicate symbols $P(x_1,\ldots,x_n)$ of $\Kc$;			
		\item $f_\iota(x_1,\ldots,x_n,y)$ for functional symbols $f(x_1,\ldots,x_n)$ of $\Kc$.
	\end{enumerate}
	
	Here, all $f_\iota$'s are required to define graphs of some functions (modulo interpretation of equality). This definition naturally extends to the translation of any $\Kc$-formulas. Naturally, $\iota$ and $\mathfrak{B}$ give a model $\mathfrak{A}$ of the language $\Kc$ with the domain $\mathbf{D}_{\iota}/{\sim_{\iota}}$, where equivalence relation $\sim_{\iota}$ is given by $=_{\iota}(x_1,x_2)$. $\mathfrak{A}$ is called the \emph{internal model} of $\Kc$.
	
	If $\mathfrak{A}\models\mathbf{T}$, then $\iota$ is an \emph{interpretation of the theory} $\mathbf{T}$ in $\mathfrak{B}$. If for a first-order theory $\mathbf{U}$ an interpretation $\iota$ is an interpretation of $\mathbf{T}$ for any $\mathfrak{B}\models\mathbf{U}$, then $\iota$ is called an interpretation of theory $\mathbf{T}$ in $\mathbf{U}$.
	
	Two theories are called \emph{bi-interpretable} if there is an interpretation of both into the other one. We study the interpretations in the B\"uchi arithmetic $\BA_n$. We note that, as $\BA_n$ is the true theory of $(\NN;=,+,V_n)$, it suffices to consider the interpretations in $(\NN;=,+,V_n)$.

		\section{One-Dimensional Interpretations}
		
		Unlike Presburger arithmetic, in the B\"uchi case it makes no difference whether to consider one-dimensional or multi-dimensional interpretations; that is, whether the elements of the internal model are natural numbers or $k$-tuples of them. In fact, we are able encode $k$-tuples into single natural numbers as long as $k$ is known and fixed beforehand.
		
		\begin{theorem}\label{one-dimen}
				Let the structure $A$ be $m$-dimensionally interpreted in B\"uchi arithmetic $\BA_n$, $m\ge 2$, under interpretation $\iota$. Then there is a one-dimensional interpretation $\iota'$ such that the internal models of $A$ given by $\iota$ and $\iota'$ are isomorphic.
			\end{theorem}
		\begin{proof}
				We work in $\BA_n$. We will give a general construction that converts an automaton accepting a set $S\subseteq\NN^m$ into an automaton accepting a set $S'\subseteq\NN$. Specifically, the new automaton will accept numbers sequentially, instead of in $m$-tuples, by inputting all the zeroth digits, then all the first digits, and so on until all the natural numbers from the tuple are input.
				
				Assume that $\{q_0,\ldots,q_{N-1}\}$ are the states of the automaton $\Ac$, $q_0$ being the initial one.
				
				We start with the same set of states and add the following new states: for each $q_i$, we first append the new states $q_{i,0},\ldots,q_{i,n-1}$ and the rules $q_i,k\ra q_{i,k}$, $k=0,\ldots,n-1$. Next, we add states $q_{i,0,0},\ldots,q_{i,0,n-1},\ldots,q_{n-1,0},\ldots,q_{n-1,n-1}$ with the corresponding rules $q_{i,j_1},k\ra q_{i,j_1,k}$. The process of extending states continues up until we add the states of the type $q_{i,j_0,\ldots,j_{m-2}}$ for each $(m-1)$-tuple $j_0,\ldots,j_{m-2}$. All the newly added states are non-final; only $q_0$ is initial while only those original states which were final in $\Ac$ remain final.
				
				At last, for each transition rule $q_i,(j_0,\ldots,j_{m-2},j_{m-1})\ra q_l$ in $\Ac$ we add an arrow $q_{i,j_0,\ldots,j_{m-2}},j_{m-1}\ra q_l$. Clearly, the movement from $q_i$ to $q_l$ along the chain of (non-final) states $q_{i,j_0},\ldots,q_{i,j_0,\ldots,j_{m-2}}$ happens under the series of inputs $j_0,\ldots,j_{m-1}$ and corresponds exactly to the movement from $q_i$ to $q_l$ under the tuple $(j_0,\ldots,j_{m-1})$.
				
				As an example, we provide the transformation of the automaton checking the equality of two binary numbers.
				
				\begin{figure}[!h]
						\centering
						\begin{tikzpicture}[shorten >=1pt,node distance=3cm,on grid,auto]
								\node [state, initial, accepting] (q_0) {$q_0$};
								\node [state, right = of q_0] (q_1) {$q_1$};
								\path[->] (q_0) edge node [above] {$\binom{1}{0},\binom{0}{1}$} (q_1)
								edge [loop above] node {$\binom{0}{0},\binom{1}{1}$} ()
								(q_1) edge [loop above] node {$\binom{0}{0},\binom{1}{0},\binom{0}{1},\binom{1}{1}$} ();
							\end{tikzpicture}
						\caption{Automaton for $=(x,y)$ before transformation}
					\end{figure}
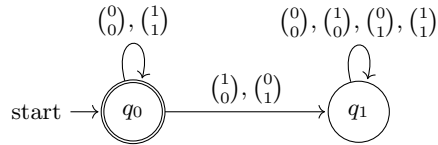
				
				\begin{figure}[!h]
						\centering
						\begin{tikzpicture}[shorten >=1pt,node distance=1.5cm,on grid,auto]
								\node [state, initial, accepting] (q_0) {$q_0$};
								\node [state, above right = of q_0] (q_00) {$q_{0,0}$};
								\node [state, below right = of q_0] (q_01) {$q_{0,1}$};
								\node [state, below right = of q_00] (q_1) {$q_1$};
								\node [state, above right = of q_1] (q_10) {$q_{1,0}$};
								\node [state, below right = of q_1] (q_11) {$q_{1,1}$};
								\path[->] (q_0) edge node [below] {$0$} (q_00)
								edge node [above] {$1$} (q_01)
								(q_00) edge node [above] {$1$} (q_1)
								edge [bend right] node [above] {$0$} (q_0)
								(q_01) edge node [above] {$0$} (q_1)
								edge [bend left] node [below] {$1$} (q_0)
								(q_1) edge node [below] {$0$} (q_10)
								edge node [above] {$1$} (q_11)
								(q_10) edge [bend right] node [above] {$0,1$} (q_1)
								(q_11) edge [bend left] node [below] {$0,1$} (q_1);
							\end{tikzpicture}
						\caption{Automaton for $=(x,y)$ after transformation}
					\end{figure}
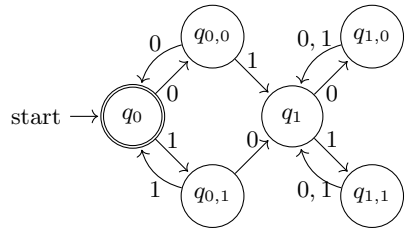
				
				Clearly, after applying this transformation procedure to each of the automata involved in the definition of the interpretation $\iota$, we obtain new automata (which are the same as $\BA_n$-definitions) of the domain and the necessary predicate and functional symbols. They define a new interpretation $\iota'$ of $A$ in $\BA_n$ obtained by the application of the bijective correspondence $\varphi\colon\NN^m\ra\NN$, where $$x=\varphi(x_0,\ldots,x_{m-1})\equivdef\Digit_n(x_i,l)=\Digit_n(x,m*l+i).$$ As $\varphi$ defines a bijection between $\NN^m$ and $\NN$, the new interpretation is isomorphic to $\iota$ by $\varphi$.
			\end{proof}
		
		We note that while the interpretation $\iota'$ thus constructed is isomorphic to $\iota$, it is not \emph{definably} isomorphic; that is, it is impossible to express the bijection $\varphi$ giving the isomorphism using a B\"uchi formula. Indeed, the following holds:
		
		\begin{lemma}
				Let $n=2$, $m=2$, $\varphi\colon\NN^2\ra\NN$ the bijection given by $x=\varphi(x_0,\ldots,x_1)\Lra\Digit_2(x_i,l)=\Digit_2(x,2l+i)$. Then $\varphi$ is not $\BA_2$-definable.
			\end{lemma}
		\begin{proof}
				Assume the contrary. Let there be an automaton checking $x=\varphi(x_0,x_1)$ for the inputs $(x,x_0,x_1)$ in binary. Let us fix the value $x_0=0$ and consider $x_1=2^k$. In this case, $x=2^{2k}$. Hence, there is an automaton $\Bc$ accepting pairs $(2^k,2^{2k})$ for all $k$. It should accept the following series of inputs: $(0,0)$ $k$ times, then $(1,0)$, then $(0,0)$ $k-1$ times, then $(0,1)$, moving to a final state after the last one (where it should remain on receiving further $(0,0)$ inputs only), while being in a non-accepting state for any other sequence of inputs.
				
				Let us denote the state of the automaton after receiving $(0,0)$ $k$ times and then $(1,0)$ as $q_k$. As the number of states is finite, we can see that for some $k<m$, $q_k=q_m$. Now consider the state of the automaton after receiving $(0,0)$ $k$ times, then $(1,0)$, then $(0,0)$ $m-1$ times, then $(0,1)$. After receiving $(1,0)$, the automaton is in $q_m$; this means that upon receiving $(0,0)$ $m-1$ times and then $(0,1)$ it must come to a final state. However, this also implies $B$ has just accepted the pair $(2^k, 2^{k+m})$, which does not belong to the accepting set of $B$, a contradiction.
			\end{proof}
		
		\section{Bi-Interpretability of $\BA_n$ and $\BA_l$}
		
		Now we prove that for all B\"uchi arithmetics, each of $\BA_n$ is bi-interpretable with $\BA_l$. This follows from the two statements below.
		
		\begin{theorem}\label{first}
				Each $\BA_{k^2}$ can be interpreted in $\BA_k$.
			\end{theorem}
		\begin{proof}
				The idea of the proof is to encode the $k^2$-ary digits as pairs of $k$-ary digits seeded into the automaton sequentially. We start from an automaton $\Ac$ performing some procedure on $k^2$-ary numbers. We apply the following transformation.
				
				We start with the original states $q_0,\ldots,q_s$ of $\Ac$. First, each state $q_i$ is supplemented by a series of non-final states called $q_{i,0},\ldots,q_{i,k-1}$ with transitions $q_i,a\ra q_{i,a}$. Then, for each transition $q_i,k*l+m\ra q_j$ of $\Ac$, we apply a new transition $q_{i,l},m\ra q_j$. Thus, each $k^2$-ary digit is entered into the new automaton as a sequence of two $k$-ary digits, $k*l+m$ encoded as $l$ followed by $m$, $l,m\in\{0,\ldots,k-1\}$.
				
				By applying the procedure described to the automata defining the predicates $=(x,y),+(x,y,z),V_{k^2}(x,y)$ in $\BA_{k^2}$, we interpret (with full domain) $\BA_{k^2}$ in $\BA_k$.
			\end{proof}
		
		\begin{theorem}\label{second}
				Each $\BA_{k}$ can be interpreted in $\BA_{k+1}$, $k\ge 2$.
			\end{theorem}
		\begin{proof}
				We start with the automata expressing $=(x,y),+(x,y,z),V_{k}(x,y)$ in $\BA_k$. For each of them, we attach a new non-final ``dead-end" state $q_t$ such that it cannot be escaped: for all $s\in\{0,\ldots,k\}$ we define $q_t,s\ra q_t$. To each state $q_i$ of the former automaton we add a new transition $q_i,k\ra q_t$: upon encountering an occurrence of digit $k$, we transfer to $q_t$ and stay there. The resulting automata in $\BA_{k+1}$ define an interpretation of $\BA_k$ with the domain consisting of all those numbers that are written without digit $k$.
			\end{proof}
		
		Now, given $\BA_k$ and $\BA_l$, we first use \sref{Theorem}{second} a sufficient number of times to interpret $\BA_k$ in some $\BA_{l^{2^p}}$ where $p$ is such that $l^{2^p}>k$, then apply \sref{Theorem}{first} $p$ times to interpret $\BA_{l^{2^p}}$ in $\BA_l$. This establishes:
		
		\begin{theorem}\label{allinterpr}
				Each $\BA_k$ is interpretable in any of $\BA_l$, $k,l\ge 2$.
			\end{theorem}

	\section*{Acknowledgments}
	
	The author thanks Lev~Beklemishev for the suggestion of the topic.


\begin{thebibliography}{10}
		
		
		
		
		
		
		\bibitem{bruyere}
		Bruy\`ere, V.
		Entiers et automates finis. M\'emoire de fin d’\'etudes, Universit\'e de Mons (1985).
		
		\bibitem{bv}
		Bruy\`ere, V., Hansel, G., Michaux, C., \& Villemaire, R.:
		Logic and p-recognizable sets of integers. Bull. Belg. Math. Soc. - Simon Stevin, 1(2), 191-238 (1994)
		
		\bibitem{bc}
		B\"uchi, J. R.:
		Weak second‐order arithmetic and finite automata. Mathematical Logic Quarterly, 6(1-6) (1960)
		
		
		
		\bibitem{haase}
		Haase, C., \& Rózycki, J.:
		On the expressiveness of Büchi arithmetic. In International Conference on Foundations of Software Science and Computation Structures, Springer, Cham, 310-323 (2021)
		
		
		
		
		
		
		
		\bibitem{presburger}
		Presburger,\;M.:
		\"{U}ber die Vollst\"{a}ndigkeit eines gewissen Systems der Arithmetik ganzer Zahlen, in welchem die Addition als einzige Operation hervortritt. Comptes Rendus du I congr\`{e}s de Math\'{e}maticiens des Pays Slaves 92–101 (1929) -- English translation in \cite{stansifer}
		
		
		
		
		
		\bibitem{stansifer}
		Stansifer,\;R.:
		Presburger's Article on Integer Arithmetic: Remarks and Translation (Technical Report). Cornell University (1984)
		
		
		\bibitem{tarskimostowski}
		Tarski\;A., Mostowski\;A., Robinson\;R.\,M.:
		Undecidable Theories. Studies in logic and the foundations of mathematics. North-Holland (1953)
		
		
		
		
	\end{thebibliography}
\end{document}